\documentclass[12pt,leqno,twoside]{amsart}

\usepackage[latin1]{inputenc}
\usepackage[T1]{fontenc}
\usepackage[colorlinks=true, pdfstartview=FitV, linkcolor=blue, citecolor=blue, urlcolor=blue]{hyperref}
\usepackage{amstext,amsmath,amscd, bezier,indentfirst,amsthm,amsgen,enumerate, geometry}
\usepackage[all,knot,arc,import,poly]{xy}
\usepackage{amsfonts,color, soul}  
\usepackage{amssymb}
\usepackage{latexsym}
\usepackage{epsfig} 
\usepackage{graphicx}
\usepackage{srcltx}
\usepackage{enumitem}
\usepackage[new]{old-arrows}
\usepackage{quiver}
\usepackage{mathtools}

\usepackage{tikz-cd}
\usepackage[all]{xy}
\usepackage{comment}

\usepackage{tikz} 
\usetikzlibrary{through} 
\usetikzlibrary{patterns} 
\usetikzlibrary{intersections} 
\usetikzlibrary{matrix} 
\usetikzlibrary{cd} 

\newtheorem{theorem}{Theorem}[section]
\newtheorem{corollary}[theorem]{Corollary}
\newtheorem{lemma}[theorem]{Lemma}
\newtheorem{proposition}[theorem]{Proposition}
\theoremstyle{definition}
\newtheorem{definition}[theorem]{Definition}
\theoremstyle{remark}
\newtheorem{remark}[theorem]{\sc Remark}
\newtheorem{example}[theorem]{\sc Example}

\newcommand{\Sing}{{\rm{Sing\hspace{2pt}}}}

\newcommand{\Disc}{{\rm{Disc\hspace{2pt}}}}

\renewcommand{\d}{{\rm{d}}}

\newcommand{\m}{\setminus}

\newcommand{\bR}{{\mathbb R}}
\newcommand{\bC}{{\mathbb C}}

\newcommand{\R}{\mathbb{R}}
\newcommand{\C}{\mathbb{C}}

\title{On the topology of the Milnor boundary for real analytic singularities}
\author{R. Ara\'ujo dos Santos}
\author{A. Menegon}
\author{M. Ribeiro}
\author{J. Seade}
\author{I. D. Santamaria Guar\'in}

\begin{document}
\maketitle

\begin{abstract}
We study the topology of the boundaries of the Milnor fibers of real analytics map-germs $f: (\R^M,0) \to (\R^K,0)$ and $f_{I}:=\Pi_{I}\circ f : (\R^M,0) \to (\R^I,0)$ that admit Milnor's tube fibrations, where $\Pi_{I}:(\bR^K,0)\to (\bR^{I},0)$ is the canonical projection for $1\leq I<K.$ For each $I$ 
we prove that the Milnor boundary $\partial F_{I}$ is given by the double of the Milnor tube fiber $F_{I+1}.$ We prove that  if $K-I\geq 2$, then  the pair $(\partial F_{I},\partial F_{f})$ is a generalized $(K-I-1)$-open-book decomposition with binding $\partial F_{f}$ and page $F_{f}\m \partial F_{f}$ - the interior of the Milnor fibre $F_{f}$ (see the definition below). This allows us to prove several new Euler characteristic formulae connecting the Milnor boundaries $\partial F_{f},$ $\partial F_{I},$ with the respectives links $\mathcal{L}_{f}, \mathcal{L}_{I},$ for each $1\leq I<K,$ and a L\^e-Greuel type formula for the Milnor boundary. 

\end{abstract}

\section{Introduction}

One of the most active and challenging areas in singularity theory is the study of non-isolated singularities of complex spaces. For instance, if $f: (\C^n,0) \to (\C,0)$ is a holomorphic germ of function with non-isolated critical point, the degeneration process of the non-critical levels to the non-isolated singularity hypersurface defined by $f$ is still not well-understood, unlike the isolated singularity case. 

One approach to this problem is to study such degeneration over a small sphere around the origin. In other words, one tries to understand the topology of the boundary of the Milnor fiber and how it degenerates to the link of $f$. This problem has been attacked by several authors like Siersma \cite{Siersma1, Siersma2}, Nemethi-Szilard \cite{NS}, Michel-Pichon \cite {MP, MP3, MPW}, Bobadilla-Menegon \cite{FM}, Menegon-Seade \cite{MS} and Aguilar-Menegon-Seade \cite{AMS}.

The corresponding understanding for real analytic singularities is still very poor. Although one can define a Milnor fibration for many classes of real analytic germs of mapping $f: (\R^M,0) \to (\R^K,0)$, not much is known about the topology of the corresponding Milnor fiber or the link of $f$ (see \cite{Oka, Sa, MS, MMS, Sa1} for some results), and even less about the boundary of such objects. 

The first part of this paper aims to introduce a new perspective to deal with such problem, inspired mainly by \cite{DR, ADD, MS2}. The idea is to relate the topology of the boundary of the Milnor fiber of $f$, denoted by $\partial F_f$, with the boundary of the Milnor fiber of the composition $f_I$ of $f$ with some projection $(\R^K,0) \to (\R^I,0)$, which we denote by $\partial F_I$. As a result, in Section \ref{op} we prove that for $K-I\geq 2$ there is a generalized open-book decomposition
$$\dfrac{f_{K-I}}{\|f_{K-I}\|}: \partial F_{I}\m \partial F_{f}\to S^{K-I-1} \, ,$$
where $f_{K-I}$ is the composition of $f$ with the projection $(\R^K,0) \to (\R^{K-I},0)$. The particular case $0\leq K-I \leq 1$ and the case of a complex ICIS $(\bC^{M+K},0)\to (\bC^{K},0)$ are analyzed. 

On the other hand, the understanding of the topology of the boundary of the Milnor fiber of the function-germ $f_I$ also provides a tool to better understanding the topology of the Milnor fiber of the map-germ $f$ itself. In fact, in Section \ref{section_EC} we use the aforementioned open-book decomposition to obtain some formulae relating the Euler characteristics of $\partial F_I$, $F_f$ and the link $\mathcal L_I$ of $f_I$, for $I=1, \dots, K$. 

Finally, in the last section of the article we use those Euler characteristic formulae to get a hint on the possible topological behaviour of real analytic map-germs on an odd number of variables and how similar or different it can be when compared with the complex setting. 



\section{Notations and basics definitions} \label{section_1}

Let $f:(\bR^{M},0)\to (\bR^{K},0), f=(f_{1},\ldots ,f_{K})$ be an analytic map germ and consider the following diagram 
\begin{equation}\label{diag}
\begin{tikzcd}[row sep = large]
& (\bR^{M},0) \arrow[d, "f"] \arrow[ddr, bend left,"f_{K-I}"]   \arrow[ddl, bend right,"f_{I}",swap] & \\
& (\bR^{K},0) \arrow[dr, "\Pi_{K-I}"]  \arrow[dl,swap, "\Pi_{I}"] & \\
 (\bR^{I},0) & & (\bR^{K-I},0) 
\end{tikzcd}
\end{equation}
where the projections $\Pi_{I}(y_{1},\ldots, y_{K})=(y_{1},\ldots, y_{I})$ and $\Pi_{K-I}(y_{1},\ldots, y_{K}):=(y_{I+1},\ldots, y_{K}),$ $f_{I}=\Pi_{I} \circ f$ and $f_{K-I}=\Pi_{K-I}\circ f.$

\vspace{0.5cm}

{\bf Basic notations and definitions:} 

\vspace{0.2cm}

The {\bf zero locus} of $f$ is defined and denoted by $V(f):=\{f=0\},$ respectively, $V(f_{I})=\{f_{I}=0\}$ and $V(f_{K-I})=\{f_{K-I}=0\}.$ Hence, $$V(f_{I})\supseteq V(f) \subseteq V(f_{K-I}).$$

\vspace{0.2cm}

The {\bf singular set} of $f$, denoted by $\Sing f$, is defined to be the set of points $x\in (\bR^{M},0)$ such that the rank of the Jacobian matrix $df(x)$ is lower than $K$. Analogously, we define the singular sets $\Sing f_{I}$ and $\Sing f_{K-I}$ of $F_I$ and $F_{K-I}$, respectively. The {\bf discriminant set} of $f$ is then defined by 
$$\Disc f:= f(\Sing f) \, .$$

\vspace{0.2cm}

The {\bf polar set} of $f$ relative to $g(x):=\|x\|^2$ is defined and denoted by $\Sing (f,g)$. Analogously, we define $\Sing (f_{I},g)$ and $\Sing (f_{K-I},g)$.

\vspace{0.2cm}

The next diagram relates the singular and the polar sets:



\begin{equation}\label{diagram1}
    \begin{tikzcd}
	{\Sing(f_I)  } & {\Sing(f_I,g) } & {\Sing\left(\dfrac{f_I}{\|f_I\|},g\right)} \\
	{\Sing(f) } & {\Sing(f,g) } & { \Sing\left(\dfrac{f}{\|f\|},g\right)} \\
	{\Sing(f_{k-I}) } & {\Sing(f_{k-I},g) } & { \Sing\left(\dfrac{f_{k-I}}{\|f_{k-I}\|},g\right)} \\
	&&&&&& {}
	\arrow[hook, from=1-1, to=1-2]
	\arrow[hook', from=1-3, to=1-2]
	\arrow[hook, from=1-1, to=2-1]
	\arrow[hook, from=1-2, to=2-2]
	\arrow[hook, from=1-3, to=2-3]
	\arrow[hook', from=3-1, to=2-1]
	\arrow[hook', from=3-2, to=2-2]
	\arrow[hook', from=3-3, to=2-3]
	\arrow[hook, from=2-1, to=2-2]
	\arrow[hook', from=2-3, to=2-2]
	\arrow[hook', from=3-3, to=3-2]
	\arrow[hook, from=3-1, to=3-2]
\end{tikzcd}
\end{equation}

\begin{definition} \label{tame} We say that a map germ $f:(\bR^{M},0)\to (\bR^K,0), f=(f_{1},\ldots ,f_{K})$ is {\it tame}, or satisfies the transversality condition at the origin if 
$$\overline{\Sing (f,g)\m V(f)}\cap \Sing f \subseteq \{0\}$$
as a germ of set at the origin.

\end{definition}

\begin{lemma}
Let $1\leq I \leq K-1.$ If $f$ is tame, then $f_{I}$ and $f_{K-I}$ are tame as well. 
\end{lemma}

It is well known that the tameness conditions for $f,$ $f_{I}$ and $f_{K-I}$ induce the following fibrations on the boundary of the closed ball $S_{\epsilon}^{M-1}:=\partial B_{\epsilon}^{M}$:

\begin{equation} \label{btub1}
f_{|}: S_{\epsilon}^{M-1}\cap f^{-1}(B_{\eta_{1}}^{K}\m \{0\})\to B_{\eta_{1}}^{K}\m \{0\} 
\end{equation}

\begin{equation} \label{btub2}
{f_{I}}_{|}: S_{\epsilon}^{M-1}\cap f^{-1}_{I}(B_{\eta_{2}}^{I}\m \{0\})\to B_{\eta_{2}}^{I}\m \{0\} 
\end{equation}

\begin{equation} \label{btub3}
{f_{K-I}}_{|}: S_{\epsilon}^{M-1}\cap f^{-1}_{K-I}(B_{\eta_{3}}^{K-I}\m \{0\})\to B_{\eta_{3}}^{K-I}\m \{0\} 
\end{equation}

Moreover, under the extra conditions $\Disc f =\{0\}$ there also exists the {\bf Milnor tube's fibration} is the following sense: there exists $ \epsilon_{0} >0$ small enough such that for all $0< \epsilon \leq \epsilon_{0}$ there exists $0<\eta_{1} \ll \epsilon$ such that the restriction map 

\begin{equation} \label{mtub1}
f_{|}: B_{\epsilon}^{M}\cap f^{-1}(B_{\eta_{1}}^{K}\m \{0\})\to B_{\eta_{1}}^{K}\m \{0\} 
\end{equation}
is a locally trivial smooth fibration, where 
$B_{\epsilon}^{M},$ respectively $B_{\epsilon}^{K},$ stand for the closed ball in $\bR^{M}$ with radius $\epsilon,$ centered at origin, respectively in $\bR^{K}$ with radius $\eta.$ 

\vspace{0.2cm}

Hence, for the same reason, we conclude the existence of the Milnor tube fibrations for $f_{I}$ and $f_{K-I}:$

\begin{equation} \label{mtub2}
{f_{I}}_{|}: B_{\epsilon}^{M}\cap f^{-1}_{I}(B_{\eta_{2}}^{I}\m \{0\})\to B_{\eta_{2}}^{I}\m \{0\} 
\end{equation}

\begin{equation} \label{mtub3}
{f_{K-I}}_{|}: B_{\epsilon}^{M}\cap f^{-1}_{K-I}(B_{\eta_{3}}^{K-I}\m \{0\})\to B_{\eta_{3}}^{K-I}\m \{0\} 
\end{equation}

From now on denote by $F_{f}$, $F_{I}$ and $F_{K-I}$ the Milnor fibers of the fibrations \eqref{mtub1}, \eqref{mtub2} and \eqref{mtub3}, respectively, by $\partial F_{f}$, $\partial F_{I}$ and $\partial F_{K-I}$ the fibers of \eqref{btub1}, \eqref{btub2} and \eqref{btub3}.

\vspace{0.2cm}

Consider the Milnor tube fibration  ${f_{I}}_{|}: B_{\epsilon}^{M}\cap f^{-1}_{I}(B_{\eta_{2}}^{I}\m \{0\})\to B_{\eta_{2}}^{I}\m \{0\}$ and $z \in B_{\eta_{2}}^{I}\m \{0\}.$ Thus the fiber $F_{I}=f^{-1}(\Pi_{I}^{-1}(z)).$ 

\vspace{0.2cm}

Denote by $D^{K-I}=\Pi^{-1}_{I}(z)\cap (B_{\eta_{1}}^{K}\m \{0\})$ the $(K-I)-$dimensional closed disc given by the intersection of the fiber of projection $\Pi_{I}: \bR^{K}\to \bR^{I}$ with the closed ball $B_{\eta_{1}}^{K}\m \{0\}$ on the target space of fibration \eqref{mtub1}. 

\vspace{0.2cm}

Thus $F_{I}=f^{-1}(D^{K-I})$ and one may consider the restriction map $f: F_{I} \to D^{K-I}$ which is a smooth surjective proper submersion, and hence a smooth trivial fibration. 

\vspace{0.2cm}

Therefore, the following homeomorphism follows $F_{I} \approx F_{f}\times D^{K-I},$ which is a diffeomorphism after smoothing the corners. On the boundary of the Milnor fiber the following diffeomorphism holds true: 
\begin{equation} \label{b1}
\partial F_{I} \approx (\partial F_{f}\times D^{K-I}) \cup (F_{f} \times S^{K-I-1}).
\end{equation}

We remark that the next Proposition is in the same vein as \cite[Corollary 4]{CSS1}.

\begin{proposition} \label{pbound} Let $f: (\bR^{M},0)\to (\bR^{K},0),$ $M>K\geq 2,$ be a tame map germ with $\Disc f =\{0\}$ and for $1\leq I<K$ consider the composition map $f_{I}=\Pi_{I}\circ f$ where $\Pi_{I}: \bR^{K} \to \bR^{I}$ is the projection map. Then the boundary $\partial F_{I}$ of the Milnor fiber $F_{I}$ is obtained (up to homeomorphism) by the gluing together two disjoint copies of the Milnor fiber $F_{I+1}$ along the common boundary $\partial F_{I+1}.$
\end{proposition}

\proof The proof follows from the composition 

\[
  \begin{tikzcd}
    (\bR^{M},0) \arrow{r}{f_{I+1}} \arrow[swap]{dr}{f_{I}=\widehat{\Pi}_{I}\circ f_{I+1}} & (\bR^{I+1},0) \arrow{d}{\widehat{\Pi}_{I}} \\
     & (\bR^{I},0)
  \end{tikzcd}
\] where $\widehat{\Pi}_{I}(y_{1},\ldots,y_{I+1})=(y_{1},\ldots,y_{I})$ and the fact that $f$ being tame implies the same to $f_{I+1}$ and $f_{I}.$ In the same manner $\Disc f=\{0\}$ implies $\Disc f_{I+1}=\{0\}$, $\Disc f_{I}=\{0\}.$ 

\vspace{0.2cm}

Hence, $F_{I} \approx F_{I+1}\times [-1,1]$ and the boundary 
 
\begin{equation} \label{b2}
\partial F_{I}\approx (\partial F_{I+1} \times [-1,1]) \cup (F_{I+1} \times \{-1,1\})
\end{equation}

\vspace{0.2cm}

\begin{figure}[h]
	\centering
	\includegraphics[width=0.5\linewidth]{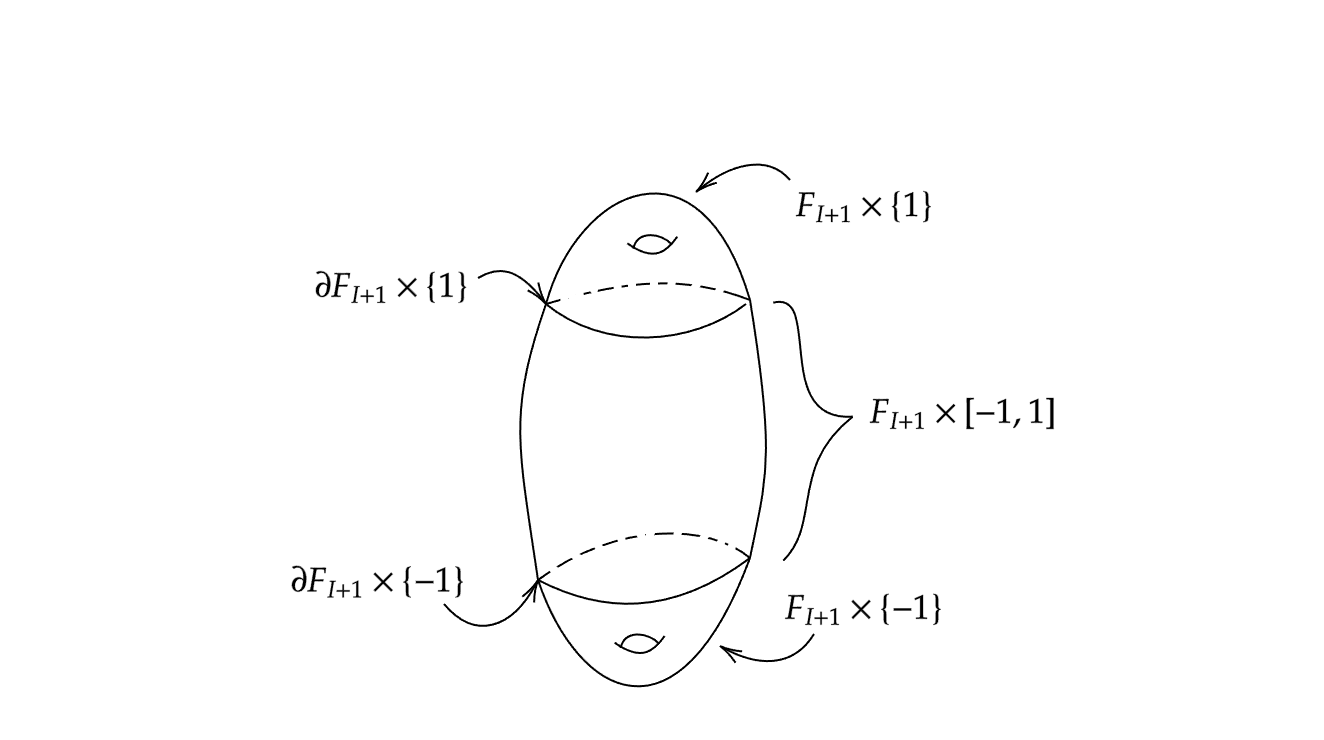}
	\caption{}
	\label{fig:f1}
\end{figure}

\vspace{0.2cm}

Now it is easy to see that the closed manifold $\partial F_{I}$ is obtained by the gluing the two disjoint copies of $F_{I+1}$ given by $F_{I+1}\times \{-1\}\cup F_{I+1}\times \{1\}$ along  the boundaries of the cylinder $\partial F_{I+1}\times [-1,1].$ See the Figure \eqref{fig:f1}. Therefore the result follows. \endproof

\section{Fibration structure on the boundary of the Milnor fiber} \label{op}

From now on we will consider $f$ tame, $\Disc f =\{0\},$ and $V(f)\neq \{0\}.$ One may adjust the radii $\eta_{1},$ $\eta_{2}$ and $\eta_{3}$ in the fibrations \eqref{btub1}, \eqref{btub2} and \eqref{btub3} such that $\partial F_{I} \subset S_{\epsilon}^{M-1}\cap f^{-1}(B_{\eta_{1}}^{K}\m \{0\})$ and the restriction map $f_{K-I}:\partial F_{I} \to \bR^{K-I}$ is well defined, for any $1\leq I < K.$

\begin{lemma}\label{sub} The restriction map $f_{K-I}:\partial F_{I} \to \bR^{K-I}$ is a smooth submersion.
\end{lemma}

\proof  For $y\in B_{\eta_{2}}^{I}\m \{0\}$ consider $\partial F_{I}= S_{\epsilon}^{M-1}\cap f_{I}^{-1}(y)$ and the matrix 
$$A(x):=\left[ \begin{array}{c}
 	 	\d f_{I}(x) \\ 
         \d f_{K-I}(x) \\
 	 	\d g(x)
\end{array}\right].$$

\vspace{0.2cm}

By the tameness of $f$ we have that for all $x\in \partial F_{I}$ the rank of $A(x)$ is maximal. Hence $f_{K-I}$ is a smooth submersion. \endproof

Now since $0\in \bR^{K-I}$ is a regular value of $f_{K-I}:\partial F_{I} \to \bR^{K-I}$ by the compactness of $\partial F_{I}$ one may choose $\tau>0$ small enough and a closed disc $D^{K-I}_{\tau} \subset \bR^{K-I}$ centered at the origin $0\in \bR^{K-I},$ such that all $y\in D^{K-I}_{\tau}$ is a regular value of the restriction map $f_{K-I}.$ 
\vspace{0.2cm}

Hence the restriction map $f_{K-I}:\partial F_{I} \cap f^{-1}_{K-I}(D^{K-I}_{\tau}) \to D^{K-I}_{\tau}$ is a smooth onto submersion, then a trivial fibration with the fiber diffeomorphic to $\partial F_{f}=\partial F_{I}\cap f_{K-I}^{-1}(0).$ Therefore, 

\begin{equation}\label{d1}
     \partial F_{I} \cap f^{-1}_{K-I}(D^{K-I}_{\tau}) \approx \partial F_{f} \times D_{\tau}^{K-I}.
\end{equation}

Denote by $T_{\tau}(\partial F_{f}):= \partial F_{f} \times D_{\tau}^{K-I}$ the closed tubular neighbourhood of the embedded submanifold $\partial F_{f} \longhookrightarrow \partial F_{I}.$ See the Figure \eqref{fig:f2}.

\vspace{0.2cm}

\begin{figure}[h]
	\centering
	\includegraphics[width=0.8\linewidth]{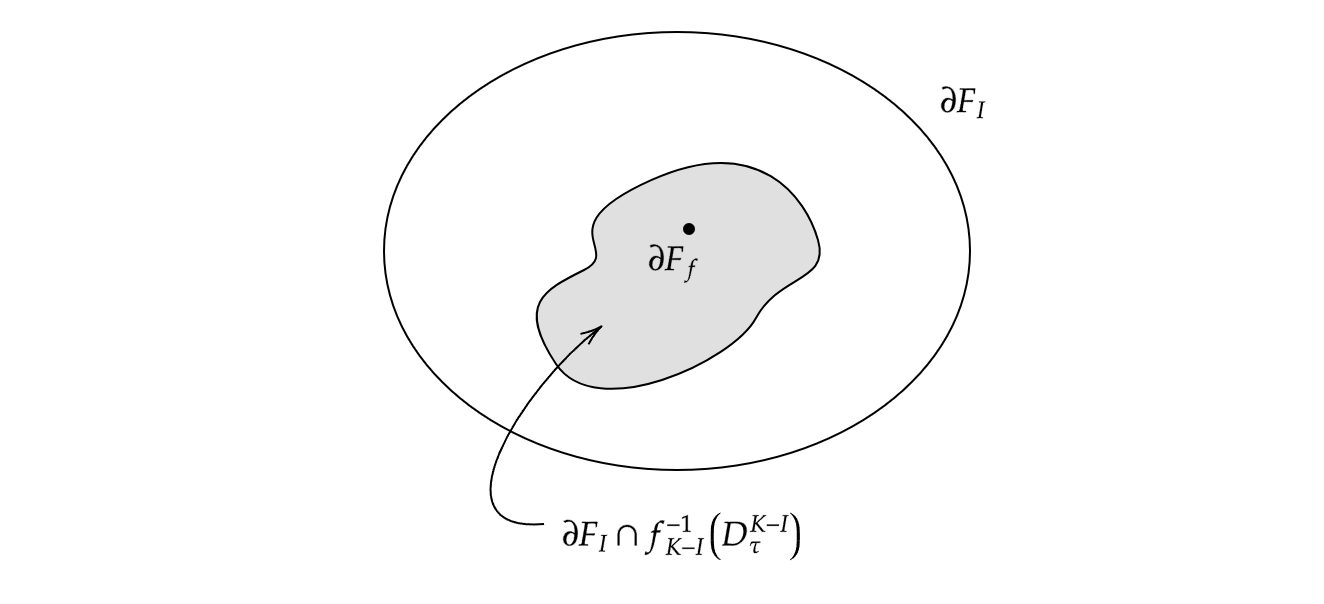}
	\caption{}
	\label{fig:f2}
\end{figure}

\vspace{0.2cm}

Then, by \eqref{b1} it follows that the complement 

\begin{equation}\label{d2}
\partial F_{I} \m  int (T_{\tau}(\partial F_{f})) \approx F_{f} \times S^{K-I-1}
\end{equation}

In this section we will prove that 
the embedded submanifold $\partial F_{f} \longhookrightarrow \partial F_{I}$ yields on the boundary  $\partial F_{I}$ an interesting structure. For that, let $f:(\bR^{M},0)\to (\bR^{K},0), M>K\geq 2$ and take $1 \leq I \leq K-2$ as in the beginning of section \ref{section_1}.    

\vspace{0.2cm}

We are now ready to introduce the main result of this section. Before that, we introduce an appropriate definition which fits with the type of fibration structure we are able to prove on the boundary of the Milnor fibration. 

\vspace{0.2cm}

Following H. Winkelnkemper in \cite{W}, A. Ranicki \cite{R}\footnote{including the appendix "The history and applications of open books", by H. E. Winkelnkemper}, E. Looijenga \cite{Lo2}, see also \cite{DACA} and \cite[section 3]{AHSS}, given $M$ a smooth manifold $M$ and $N\subset M$ a submanifold of codimension $k\geq 2$ in $M,$ suppose that for some trivialization $t:T(N)\to N\times B^{k}$ of a tubular neighbourhood $T(N)$ of $N$ in $M,$ the fiber bundle defined by the composition
$\pi \circ t$ in the diagram below

\[
  \begin{tikzcd}
    T(N)\m N \arrow{r}{t} \arrow[swap]{dr}{\pi \circ t} & N\times (B^{k}\m \{0\}) \arrow{d}{\pi} \\
     & S^{k-1}
  \end{tikzcd}
\]
where $\pi (x,y):=\dfrac{y}{\|y\|},$ extends to a smooth locally trivial fiber bundle $p: M\m N \to S^{k-1};$ e.i., $p_{|_{T(N)\m N}}= \pi \circ t.$ 

\vspace{0.2cm}

In such a case the pair $(M,N)$ above will be called a {\it generalized $(k-1)$-open-book decomposition on $M$ with {\it binding $N$}} and {\it page} the fiber $p^{-1}(y), y\in S^{k-1}.$ 

\vspace{0.2cm}

The main result of this section is:

\begin{theorem}\label{bopen} Let $f: (\bR^{M},0)\to (\bR^{K},0),$ $M>K\geq 2,$ be a tame map germ with $\Disc f =\{0\}.$ Then, for each $1\leq I<K$ such that $K-I\geq 2,$ the pair $(\partial F_{I},\partial F_{f})$ is a generalized $(K-I-1)$-open-book decomposition, with binding $\partial F_{f}$ and page $F_{f}\m \partial F_{f}$- the interior of the Milnor fiber $F_{f}.$
\end{theorem}

\proof Consider the restriction map $f_{K-I}:\partial F_{I} \to \bR^{K-I}.$ It follows by Lemma \ref{sub} that we may adjust the radii of the fibrations \eqref{btub1}, \eqref{btub2} and \eqref{btub3} such that for a small enough radius $\tau$ in the diagram below 

\begin{equation}\label{p1}
 \begin{tikzcd}
    \partial F_{I} \cap f^{-1}_{K-I}(D^{K-I}_{\tau}-\{0\}) \arrow{r}{f_{K-I}} \arrow[swap]{dr}{\dfrac{f_{K-I}}{\|f_{K-I}\|}} & D^{K-I}_{\tau}-\{0\} \arrow{d}{\Pi_{R}(z)=\frac{z}{\|z\|}} \\
     & S^{K-I-1}
  \end{tikzcd}
\end{equation}
the projection $\dfrac{f_{K-I}}{\|f_{K-I}\|}$ is a (trivial) locally fiber bundle, where $\Pi_{R}$ is the radial projection. 

\vspace{0.2cm}

It induces the trivial fibration on the diagonal projection

\begin{equation}\label{p2}
 \begin{tikzcd}
    \partial F_{I} \cap f^{-1}_{K-I}(S^{K-I-1}_{\tau}) \arrow{r}{f_{K-I}} \arrow[swap]{dr}{\dfrac{f_{K-I}}{\|f_{K-I}\|}} & S^{K-I-1}_{\tau} \arrow{d}{\Pi_{R}(z)=\frac{z}{\|z\|}} \\
     & S^{K-I-1}
  \end{tikzcd}
\end{equation}

Applying again the Lemma \ref{sub} in the diagram below we get that the horizontal map is a locally trivial fibration over its image; and thus, the diagonal projection is again a locally trivial smooth fibration.  

\begin{equation}\label{p3}
 \begin{tikzcd}
    \partial F_{I} \m f^{-1}_{K-I}(int(D^{K-I}_{\tau})) \arrow{r}{f_{K-I}} \arrow[swap]{dr}{\dfrac{f_{K-I}}{\|f_{K-I}\|}} & \bR^{K-I} \m \{0\}\arrow{d}{\Pi_{R}(z)=\frac{z}{\|z\|}} \\
     & S^{K-I-1}
  \end{tikzcd}
\end{equation}

Now we may glue the fibrations \eqref{p1} and \eqref{p3} along the fibration \eqref{p2} to get a smooth projection of a locally trivial fiber bundle 

\begin{equation}\label{p4}
\dfrac{f_{K-I}}{\|f_{K-I}\|}: \partial F_{I}\m \partial F_{f}\to S^{K-I-1}.
\end{equation}

\vspace{0.2cm}

Now see that the diffeomorphism of \eqref{d1} says that the trivialization in the horizontal map of the diagram \eqref{b1} is given by $$\partial F_{I} \cap f^{-1}_{K-I}(D^{K-I}_{\tau}-\{0\}) \approx \partial F_{f} \times (S_{\tau}^{K-I-1}\times (0,\tau]).$$ Hence
the fiber of the diagonal fibration in the diagram \eqref{p1} should be diffeomorphiuc to $\partial F_{f}\times (0,\tau].$ 

\vspace{0.2cm}

\begin{figure}
	\centering
	\includegraphics[width=0.7\linewidth]{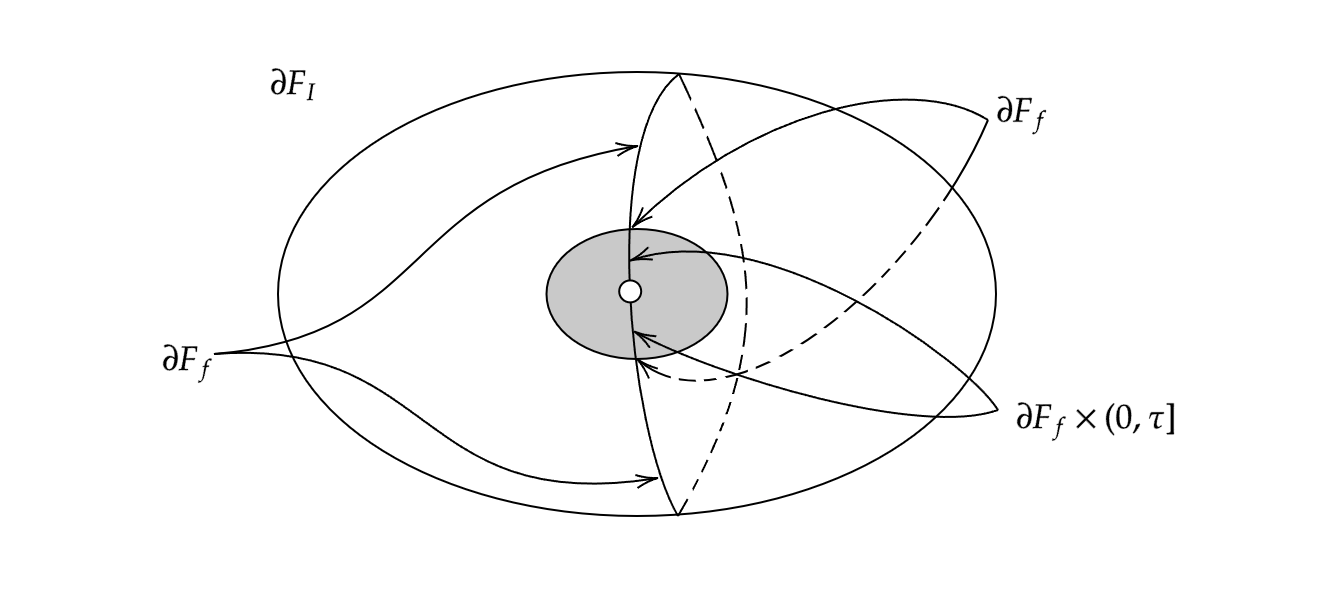}
	\caption{}
	\label{fig:f3}
\end{figure}

On the other hand, the diffeomorphism \eqref{d2} assures that the fiber in the diagonal projection of the diagram \eqref{p3} should be diffeomorphic to $F_{f}.$ The fiber of the boundary trivial fibration \eqref{p2} is clearly diffeomorphic to $\partial F_{f}.$ Therefore, we conclude that the fiber of fibration \eqref{p4} must be diffeomorphic to the gluing (using the identity diffeomorphism on the boundary) $F_{f}\cup_{\partial F_{f}}(\partial F_{f}\times (0,\tau])=F_{f}\m \partial F_{f}.$ See Figure \eqref{fig:f3} and the proof is finished for $K-I-1\geq 1,$ i.e., $K-I\geq 2.$ \endproof 

\begin{remark} Notice that for $K-I=2$, a generalized open-book decomposition is an open-book in the usual sense. We also remark that
in  Theorem \ref{bopen} we assumed $K-I\geq 2;$ If we consider the case $K=I$ then by convention $f_{K-I}:=f_{0}\equiv 0$ and $\bR^{0}=\{0\},$ so there is nothing  to be said. For $I=K-1,$ then the study of the restriction function $f_{K-I}:\partial F_{I}\to \mathbb{R}$ reduces to that of Proposition \ref{pbound} and the construction above leads to a "fibration" over $S^{0}=\{-1,1\}.$ \end{remark}

\subsection{The case of ICIS holomorphic map germ}

Let us consider now a holomorphic map germ $f=(f_{1},\ldots, f_{K}):(\bC^{M+K},0)\to (\bC^{K},0), K\geq 2.$ For $1\leq I < K$ consider the complex projections $\Pi_{I}:(\bC^{K},0)\to (\bC^{I},0), \Pi_{I}(z_{1},\ldots,z_{K})=(z_{1},\ldots,z_{I}),$ and $\Pi_{K-I}:(\bC^{K},0)\to (\bC^{K-I},0), \Pi_{K-I}(z_{1},\ldots,z_{K})=(z_{I+1},\ldots,z_{K}).$ Thus, the compositions as in the diagram \eqref{diag} becomes $f_{I}:=\Pi_{I}\circ f = (f_{1},\ldots, f_{I})$ and $f_{K-I}:=\Pi_{K-I}\circ f = (f_{I+1},\ldots, f_{K}).$

\vspace{0.2cm}

If we assume further that $f$ is ICIS,
it is known that $\Disc f:=f(\Sing f)$ is a complex hypersurface in $\bC^{K},0$ and then for all $\delta>0$ small enough the space $B_{\delta}^{2K}\m \Disc f,$ where $B_{\delta}^{2K}$ stand for the open ball in $\bC^{K}\equiv \bR^{2K}.$ Then, the space $B_{\delta}^{2K}\m \Disc f$ is a connected space. 

\vspace{0.2cm}

In fact, it was proved by H. Hamm, L\^e D. Tr\`ang and by E. Looijenga in \cite{Lo} that there exists $\epsilon_{0}>0$ small enough such that for each $0<\epsilon \leq \epsilon_{0}$ there exists $0<\delta \ll \epsilon$ such that the projection map $$f:\overline{B}_{\epsilon}^{2M+2K}\m f^{-1}(\Disc f)\to B_{\delta}^{2K}\m \Disc f$$ is a smooth locally trivial fibration. Thus, by the connectedness property of the base space the Milnor fiber $F_{f}$ is uniquely defined, up to diffeomorphism.

\vspace{0.2cm}

In addition to that, the ICIS condition is equivalent to the condition $ \Sing f \cap V_{f}=\{0\}.$ Hence, $f$ is tame according to the Definition \ref{tame}. We may also use the argument of Looijenga in \cite{Lo} for a "good representative" to guarantee that, up to a linear coordinate change in $\bC^{K},$ it follows that either map germ $f_{I}$ and $f_{K-I}$ are ICIS as well. 

\vspace{0.2cm}

The following result is an interesting application of our Theorem \ref{bopen} and also it provides an extension of \cite[Proposition 3.2, p. 481]{C}. See also \cite[Chapter 3]{NS} and compare with \cite[Proposition 3.2, p. 481]{C}.

\begin{proposition}\label{contact-cplx} Let $f:(\C^{M+K},0) \to (\bC^{K},0),  K > I \geq 1,$ be a germ of ICIS such that $f_{I}$ and $f_{K-I}$ are ICIS as well. Then, the Milnor projection  $\dfrac{f_{K-I}}{\|f_{K-I}\|}: \partial F_{I}\m \partial F_{f}\to S^{2K-2I-1}$ induces a generalized $(2K-2I-1)$-open-book decomposition on the boundary $\partial F_{I}$ with binding $\partial F_{f}.$
\end{proposition}

\begin{proof} Since $K>I$ then $K-I\geq 1$ and the dimension of the sphere on the target space is $2(K-I)-1\geq 1.$ Hence,  the same ideas in the proof of Theorem \ref{bopen} work in this case.
\end{proof}

\begin{remark}

\begin{enumerate}
\item [(1)] We recall that if one has an isolated complex hypersurface singularity germ, then its link has a canonical contact structure which is Stein fillable (see for instance \cite{PPP}). These statements extend naturally to the setting we envisage in Proposition \ref{contact-cplx}.

\item [(2)] Still the complex ICIS above, the generalized open-book decomposition also extends with the same proof for the pair links  $(\mathcal{L}_{I}, \mathcal{L}_{f}),$ where $\mathcal{L}_{I}:=f_{I}^{-1}(0)\cap S_{\epsilon}^{M-1}$ and  $\mathcal{L}_{f}:=f^{-1}(0)\cap S_{\epsilon}^{M-1},$ for all $\epsilon>0$ smal enough. 

\end{enumerate}

\end{remark}

\section{The Euler characteristic formulae} \label{section_EC}

Let $f:(\bR^{M},0)\to (\bR^{K},0),$ $M>K\geq 2,$ be an analytic map-germ. We will assume along the section that $f$ is tame and $\Disc f =\{0\}.$ Thus, $\dim F_{f}=M-K$ and $\dim \partial F_{f}=M-K-1.$ 

\vspace{0.2cm}

Denote by $\widehat{F_{f}}=F_{f}\cup_{\partial F_{f}}F_{f}$ the closed manifold built by gluing two copies of $F_{f}$ along the boundary $\partial F_{f}$  using the identity diffeomorphism on $\partial F_{f}.$ By the additive property of the Euler characteristic we have that $\chi(\widehat{F_{f}})=2\chi(F_{f})-\chi(\partial F_{f}).$ Hence

\begin{equation}\label{euler}
\chi(\partial F_{f}) = 
     \begin{cases}
       0, &\quad\text{if M-K is even}.\\
       2\chi(F_{f}), & \quad\text{if M-K is odd}. 
     \end{cases}
\end{equation}

Applying again the additive Euler characteristic to the diffeomorphism \eqref{b1} we get 

$$\chi(\partial F_{I})=\chi(\partial F_{f}\times D^{K-I}) + \chi(F_{f} \times S^{K-I-1})-\chi (\partial F_{f} \times S^{K-I-1})=
$$

$$=\chi(\partial F_{f}) + \chi(F_{f}).\chi (S^{K-I-1})-\chi (\partial F_{f}).\chi( S^{K-I-1}).$$

\vspace{0.2cm}

Thus, together with \eqref{euler} it reduces to

\begin{equation}\label{euler1}
\chi(\partial F_{I})= 
     \begin{cases}
       \chi(F_{f}).\chi (S^{K-I-1}), &\text{if M-K is even}.\\
       \chi(F_{f}).\chi (S^{K-I}), & \text{if M-K is odd}. 
     \end{cases}
\end{equation}

We may consider the convention $\chi(S^{-1})=\chi(\emptyset)=0$ and the fact that for the $0-$dimensional sphere $\chi(S^{0})=\chi(\{-1,1\})=2.$ Then all the above discussion, including the special cases of $I=1$ and $I=K$ may be summarized as below. By convention, for $I=1$ we denote $F_{1}:=F_{f}.$

\begin{theorem}\label{pbound2}
 Let $f:(\bR^{M},0)\to (\bR^{K},0),$ $M>K\geq 2,$ be an analytic map-germ, tame with $\Disc f =\{0\}.$ Then:

\begin{enumerate}
  \item [$1)$] $\chi(\partial F_{I})= \chi(F_{f}).\chi (S^{M-I-1}),$ for any $1\leq I \leq K.$ 
  \item [$2)$] {\bf L\^e-Greuel's type formula:} $\chi(\partial F_{I+1})-\chi(\partial F_{I})= 
       2(-1)^{M-I}\chi(F_{f}),$ for any $1\leq I <K.$ 
  \item [$3)$]  $\chi(\partial F_{I})=\chi(\partial F_{I+2}),$ for any $1\leq I < K-1.$    
\end{enumerate}
\end{theorem}

\proof The item $1)$ follows from the identity \eqref{euler1}. To prove the item $2)$ just exchange $I$ by $I+1$ in the item $1)$ and take the difference. The item $3)$ is immediate from item $1).$  \endproof

\subsection{Relating the Euler characteristic of the links}

Consider again $f:(\bR^{M},0)\to (\bR^{K},0),$ $M>K\geq 2,$ a tame polynomial map-germ with $\Disc f =\{0\}.$ For each $1< I \leq K$ the map $f_{I}:(\bR^{M},0)\to (\bR^{I},0)$ admits the Milnor tube fibrations \eqref{btub2} and by restriction  it induces the fibrations $f_{I}:B^{M}_{\epsilon}\cap f_{I}^{-1}(S_{\eta}^{I-1})\to S_{\eta}^{I-1}$ with Milnor fibers $F_{I}$ and the fibration $f_{I}:S^{M-1}_{\epsilon}\cap f_{I}^{-1}(S_{\eta}^{I-1})\to S_{\eta}^{I-1}$ with fiber $\partial F_{I}.$   

\vspace{0.2cm}

Denote by $T_{\eta}(F_{I}):=B^{M}_{\epsilon}\cap f_{I}^{-1}(S_{\eta}^{I})$ the Milnor tube of $f_{I}$ and by $\mathcal{L}_{I}:=f_{I}^{-1}(0)\cap S_{\epsilon}^{M-1}$ the respective link.

\vspace{0.2cm}

We may consider $\eta$ small enough such that the sphere $S_{\epsilon}^{M-1}$ is homeomorphic to the gluing $T_{\eta}(f_{I})\cup_{\partial T_{\eta}(f_{I})}N_{\eta}(f_{I}),$ where $\mathcal{L}_{I} \subset N_{\eta}(f_{I}):=f_{I}^{-1}(B^{I}_{\eta})$ is a semi-algebraic neighbourhood that retract to the link $\mathcal{L}_{I},$ as proved by A. Durfee in \cite{Durfee}.

\vspace{0.2cm}

Thus, 
$$\chi(S_{\epsilon}^{M-1})=\chi(T_{\eta}(f_{I}))+\chi(N_{\eta}(f_{I}))-\chi(\partial T_{\eta}(f_{I}))=\chi(F_{I})\chi(S^{I-1})+\chi(\mathcal{L}_{I})-\chi(\partial F_{I})\chi(S^{I-1}).$$

\vspace{0.2cm}

Hence,
\begin{equation} \label{euler2}
\chi(\mathcal{L}_{I})=\chi(S^{M-1})- \chi(F_{f})\chi(S^{I-1})+\chi(\partial F_{I})\chi(S^{I-1})
\end{equation}

\begin{lemma} \label{lemma_link}
 The following holds true: $$\chi(\mathcal{L}_{I})=\chi(S^{M-1})+(-1)^{M-I-1}\chi(F_{f})\chi(S^{I-1}).$$
\end{lemma}

\proof The proof follows from Proposition \ref{euler1} and equation \eqref{euler2}. \endproof
\vspace{0.2cm}












The next result provides in particular a second proof of \cite[Proposition 7.1, p. 4861]{DR}.

\begin{proposition}\label{pbound3} Let $f:(\bR^{M},0)\to (\bR^{K},0),$ $M>K\geq 2,$ be a tame polynomial map-germ with $\Disc f =\{0\}.$ Then: 

\begin{enumerate}

\item [$1)$] $\chi(\mathcal{L}_{I+1})- \chi(\mathcal{L}_{I})=
       2(-1)^{M-I}\chi(F_{f}),$ for each $1\leq I < K.$
\item [$2)$] $\chi(\mathcal{L}_{I+2})= \chi(\mathcal{L}_{I}),$ for each $1\leq I < K-1.$
\end{enumerate} 
\end{proposition}

\proof In the equation \eqref{euler2} just exchange $I$ by $I+1$ and take the difference. Then we have

\vspace{0.2cm}

$\chi(\mathcal{L}_{I+1})- \chi(\mathcal{L}_{I})=\chi( F_{f})\chi(S^{I-1})-\chi(\partial F_{I})\chi(S^{I-1})-\chi(F_{f})\chi(S^{I})+\chi(\partial F_{I+1})\chi(S^{I}).$

\vspace{0.2cm}

Now we may apply Proposition \ref{pbound2}, item $1),$ to get

\vspace{0.2cm}

$\chi(\mathcal{L}_{I+1})- \chi(\mathcal{L}_{I})=(-1)^{M-I}\chi(F_{f})\chi(S^{I-1})+(-1)^{M-I}\chi(F_{f})\chi(S^{I})=2(-1)^{M-I}\chi(F_{f}).$ 

\vspace{0.2cm}

This ends the proof of item $1).$ Item $2)$ is trivial. \endproof 

\begin{remark}

\item [(1)] We point out that the L\^e-Greuel type formula obtained in the Theorem \ref{pbound2}, item (2), is somehow similar to that obtained in \cite[Theorem 1, p. 3]{CSG}, but with the difference that in \cite{CSG} the authors worked with the Euler number of the Milnor fibers, instead of its boundary. 

\item [(2)] In view of the Proposition \ref{pbound2} and the Proposition \ref{pbound3}, we can see that for all $1\leq I <K$ we have  $\chi(\partial F_{I+1})-\chi(\partial F_{I})=\chi(\mathcal{L}_{I+1})-\chi(\mathcal{L}_{I}).$ Thus, $\chi(\partial F_{I+1})-\chi(\mathcal{L}_{I+1})=\chi(\partial F_{I})-\chi(\mathcal{L}_{I})=\cdots =\chi(\partial F_{2})-\chi(\mathcal{L}_{2})=\chi(\partial F_{1})-\chi(\mathcal{L}_{1}).$ Hence, it suggests the following definition.

\end{remark}

\begin{definition} 
Let $f:(\bR^{M},0)\to (\bR^{K},0),$ $M>K\geq 2,$ be a tame polynomial map-germ with $\Disc f =\{0\}$. The degree of {\it degeneracy on the Milnor boundary of $f$} is defined as the number
$$DB(f):=\chi(\partial F_{1})-\chi(\mathcal{L}_{1}) \, .$$
\end{definition}

Clearly, if $f$ has an isolated singularity at the origin one has that $DB(f)=0$. 

\section{On the boundaries of the Milnor fibers and the links on each stage I.}

In the real setting we do not expect to prove theorems regarding the degree of connectivity of the Milnor fibers, its boundaries nor the respective links of $f_{I},$ on each stage $I$. Notwithstanding, in the case where the dimension $M$ of the source space is even, for all $I,$ $1\leq I \leq K,$ we may write $\chi(\partial F_{I})=\chi(S^{I+1})\chi(F_{I})$ and as an application of Lemma \ref{lemma_link} we conclude that $\chi(\partial F_{I})=\chi(\mathcal{L}_{I}),$ and hence $DB(f)=0.$ However, if the source space $M$ is odd-dimensional some interesting relations between the boundaries of the Milnor fiber, the links of the singularities and the Milnor fibers on the Milnor tubes come up on each stage $I$, and it provides a way to distinguish between the homotopy type of the Milnor boundary and the link of the singularities $f_{I}$ for each $1\leq I \leq K,$ as described below.

\vspace{0.2cm}

We first remind that for odd dimension $M\geq 2$ the equation \eqref{euler2} becomes 
\begin{equation*}
(*):~ \chi(\mathcal{L}_{I})=2 - \chi(F_{f})\chi(S^{I-1})+\chi(\partial F_{I})\chi(S^{I-1}).
\end{equation*}

This allows us to prove the below result whose proof we left as an exercise.

\begin{lemma} \label{carac} Let $M$ odd and $I$ such that $M>K \geq I \geq 2.$ Then, for each $I$ the following conditions hold for the links $\mathcal{L}_{I}$ and the boundaries $\partial F_{I}$ of the Milnor fibers $F_{I}:$

\begin{enumerate}
    \item [(1)] if $I$ is even then $\chi(\mathcal{L}_{I})=2,$ by equation $(*).$ Moreover, since $\dim F_{I}=M-I$ is odd then $\chi(\partial F_{I})=2\chi(F_{I})=2\chi(F_{f});$ 

    \item [(2)] if $I$ is odd then $\chi(\mathcal{L}_{I})=2-2\chi(F_{f}),$ by equation $(*).$ Moreover, since $\dim \partial F_{I}=M-I-1$ is odd then $\chi(\partial F_{I})=0.$
\end{enumerate}
\end{lemma}

Now we are ready to state the main result of this section.

\begin{theorem}\label{carac2} Consider $M$ odd, $M>K \geq I \geq 2.$ Let $f:(\mathbb{R}^{M},0)\to (\mathbb{R}^{K},0)$ and $f_{I}:(\mathbb{R}^{M},0)\to (\mathbb{R}^{I},0)$ be real analytic map germs as in the diagram \eqref{diag}. Then, $\chi(F_{f})=1$ if and only if $\chi(\partial F_{I})=\chi(\mathcal{L}_{I})$ in some stage $I.$ Moreover, if the last equality holds true on any stage $I$ it also will holds true on all stages $I,$ $2\leq I \leq K < M.$ 
\end{theorem}

\begin{proof} The proof follows from Lemma \ref{carac}.

\vspace{0.2cm}

For the "if" case, we can see that $\chi(F_{I})=\chi(F_{f})=1$ implies that the two quantities $\chi(\partial F_{I})=\chi (\mathcal{L}_{I})$ in the either cases of Lemma \ref{carac}. Moreover, the equality in some stage $I$ clearly implies that on all $I,$ $2\leq I \leq K <M.$ 

\vspace{0.2cm}

For the "only if" case, if we suppose that in some stage (even or odd) $I$ the equality $\chi(\partial F_{I})=\chi (\mathcal{L}_{I})$ holds true, then again by Lemma \ref{carac} we conclude that $2\chi(F_{I})=2.$ Therefore, $\chi(F_{f})=1.$
\end{proof}

The next result provides a natural class of map germs where one of two conditions above holds true. Beside that, it also provides another proof of \cite[Proposition 3, item ii), p. 71]{ADD}.

\begin{corollary}\label{isola} Let $f:(\bR^{M},0)\to (\bR^{K},0)$, with $M>K\geq 2$ and $M$ odd, be a real analytic map-germ with an isolated critical point at the origin. Then, for each $I,$ $M > K \geq I \geq 1,$ we have that $\chi(F_I)=1.$ 
\end{corollary}

\begin{proof} The proof might be left as exercise, but we will screatch it below for the sake of convenience. 

\vspace{0.2cm}

Since $f$ have an isolated singular point at the origin, one may apply the diagram \eqref{diagram1} to get that $\Sing (f_{I})\subseteq \{0\}$ for each fixed $I.$ It is enough to consider the case $\Sing (f_{I})=\{0\},$ because the case $\Sing (f_{I})=\emptyset$ the result follows as an easy  application of the Inverse Function Theorem version for map germs.

\vspace{0.2cm}

Now, if we assume further that $M$ is odd then for each $I$ the link $\mathcal{L}_{I}$ must be not empty, and it is in fact a smooth manifold diffeomorphic to $\partial F_{I}$ and thus $\chi(\partial F_{I})=\chi(\mathcal{L}_{I}).$ Therefore, one may apply the Theorem \ref{carac2} and conclude that $\chi(F_{I})=1,$ for each $1\leq I \leq K < M.$ \end{proof}

\begin{remark}
For the existence of map germ $(\bR^{M},0)\to (\bR^{K},0),$ M odd, $M>K\geq 2,$ with isolated critical point at the origin, the reader may consult \cite[section 5.2, p. 101]{AHSS}. 
\end{remark}
\begin{corollary} Let $M > K \geq I \geq 1$ and $f$ be as in Theorem \ref{carac2}. If $\chi(F_f) \neq 1$ then at all stages $I,$ the Milnor boundary $\partial F_{I}$ and the respective link $\mathcal{L}_{I}$ of $f_{I}$ cannot be homotopically equivalent.
\end{corollary}

\begin{proof} It is now trivial because $\chi(F_{I})=\chi(F_{f})\neq 1,$ on each stage $I.$ Therefore, by Theorem \ref{carac2} the respectives Milnor boundary $\partial F_{I}$ and the link $\mathcal{L}_{I}$ can not be homotopically equivalent.  \end{proof}

The next example shows that for odd-dimension $M$ it is easy to construct a family of map germ where the Euler characteristic of the Milnor fiber is not equal to one. 

\begin{example}
Let $f:(\bR^{M},0)\to (\bR^{K},0)$ be an analytic map germ $M>K\geq 2,$ with $\Sing f =\{0\},$ and $g:(\bR^{K},0)\to (\bR^{K},0)$ be an analytic ramified covering map branched along $\{0\}$ with t-sheets, $t\geq 2.$ Then, for all fixed $z \in (\bR^{K},0),$ $0<\|z\|\ll 1,$ and all $x\in g^{-1}(z)$ the map $g$ is a local diffeomorphism and the fiber $g^{-1}(z)$ consists of a finite number of points, and we set $t:=\# g^{-1}(z).$ Thus, $\Sing g =\{0\}$ and the composition map germ $h=g\circ f: (\bR^{M},0)\to (\bR^{K},0)$ safisfies that $\Sing h =f^{-1}(0)\subseteq V_{h}.$ Since $f$ is tame, it is not hard to see that $h$ is tame as well. Then the map $h$ admits a  Milnor tube fibration with Milnor fiber $F_{h}=\sqcup_{i=1}^{t}F_{f}$ (t-disjoint copies of $F_{f}$). Therefore we have that $\chi(F_{h})=t. \chi(F_{f})=t\geq 2,$ where we use that $\chi(F_{f})=1$ by Corollary \ref{isola}. 
\end{example}

\end{document}